\documentclass[10pt]{amsart}
\usepackage{amssymb,amsthm,amsmath,amsfonts}
\usepackage{hyperref}
\usepackage{pstricks}
\usepackage{graphicx}
\usepackage{pst-node}
\def\R{{\mathbb R}}
\def\Rr{{\mathcal R}}
\def\C{{\mathbb C}}

\def\Z{{\mathbb Z}}

\def\<{\langle}
\def\>{\rangle}

\def\E{{\mathbb E}}
\def\U{{\mathbb U}}
\def\V{{\mathbb V}}

\def\I{\mathbb I}

\def\X{\mathbb X}

\def\Y{\mathbb Y}

\newcommand{\be}{\begin{equation}}
\newcommand{\ee}{\end{equation}}
     \newtheorem{theorem}{Theorem}[section]
       \newtheorem{proposition}[theorem]{Proposition}
       
      \newtheorem{lemma}[theorem]{Lemma}

\numberwithin{equation}{section}
\begin{document}
\title{DUAL LUKACS REGRESSIONS OF NEGATIVE ORDERS FOR NON-COMMUTATIVE VARIABLES}
\author[K. Szpojankowski]{Kamil Szpojankowski}
\address{Wydzia\l{} Matematyki i Nauk Informacyjnych\\
Politechnika Warszawska\\
ul. Koszykowa 75\\
00-662 Warszawa, Poland\\
k.szpojankowski@mini.pw.edu.pl}


\begin{abstract}
In the paper we study characterizations of probability measures in free probability. By constancy of regressions for random variable $\V^{1/2}(\I-\U)\V^{1/2}$ given by $\V^{1/2}\U\V^{1/2}$, where $\U$ and $\V$ are free, we characterize free Poisson and free binomial distributions. Our paper is a free probability analogue of results known in classical probability \cite{BobWes2002Dual}, where gamma and beta distributions are characterized by constancy of $\E\left((V(1-U))^{i}|UV\right)$, for $i \in \{-2,-1,1,2\}$. This paper together with previous results \cite{SzpojanWesol} exhaust all cases of characterizations from \cite{BobWes2002Dual}.
\end{abstract}
\maketitle

\section{Introduction}
 
One of the most famous characterizations in classical probability is the Lukacs theorem \cite{Lukacs}, saying that if $X$ and $Y$ are positive, independent random variables such that,
\be\label{uav}
U=\frac{X}{X+Y}\qquad \mbox{and}\qquad V=X+Y,
\ee
are independent, then $X$ and $Y$ have gamma distribution $G(p,a)$ and $G(q,a)$. Here by the gamma distribution $G(r,c)$, $r,c>0$, we understand the probability distribution with density
$$
f(x)=\frac{c^r}{\Gamma(r)}\,x^{r-1}\,e^{-cx}\,I_{(0,\infty)}(x).
$$
This result was generalized in many directions. For example, it is known that assumptions of independence of $U$ and $V$ can be weakened to constancy of regressions (see \cite{BolHark}, \cite{LahaLukacs}),
\begin{align}
\label{LL}
\E\left(\frac{X}{X+Y}\bigg|X+Y\right)=&c,\\
\E\left(\left(\frac{X}{X+Y}\right)^2\bigg|X+Y\right)=&d.\nonumber
\end{align}
In \cite{BobWes2002Dual} authors proved so called dual Lukacs regressions which says that if $U$ and $V$ are independent, such that $U$ is supported on interval $(0,1)$, and $V$ is positive and
\begin{align*}
\E\left(\left(V(1-U)\right)^{i}|UV\right)=&c,\\
\E\left(\left(V(1-U)\right)^{j}|UV\right)=&d,
\end{align*}
for one of the pairs $(i,j)\in\{(-2,-1),\, (-1,1),\, (1,2)\}$, then $U$ is beta distributed, and $V$ is gamma distributed. Note that in the Lukacs theorem we have $X=UV$ and $Y=V(1-U)$. By this one can see that the above result is dual to the characterization based on \eqref{LL}, where independence of $X$ and $Y$, and constancy of regressions of $U$ given by $V$ are assumed.

It has been noticed before that characterizations of distributions in classical and free probability can be closely related. Many of classical characterizations has counterparts in free probability. Explicit example of such analogy is the Bernstein theorem \cite{Bernstein} saying that, if $X,Y$ are independent, such that $X+Y$ and $X-Y$ are independent then $X$ and $Y$ have normal distribution. Free analogue of this theorem was proved in \cite{NicaChar}. It says that for free $\X$ and $\Y,$ if $\X+\Y$ and $\X-\Y$ are free then $\X$ and $\Y$ have Wigner (semicircular) distribution, which in free probability plays the role of normal law in classical probability.\\ We have to note that not all characterizations of distributions of independent random variables, translated to the language of free random variables, remain characterizations. Famous Cram\'er's theorem says that if $X$ and $Y$ are independent such that $X+Y$ has normal distribution then $X$ and $Y$ have normal distribution. Free analogue of this theorem turns out not to be true. In \cite{BerVoicSuperConverg} authors proved that there exist free random variables $\X$ and $\Y$ such that $\X+\Y$ has Wigner distribution, but $\X$ and $\Y$ are not Wigner distributed.
 
The Lukacs theorem and other characterizations related to it, turn out to have free counterparts (see \cite{BoBr2006}, \cite{EjsmontLL}, \cite{EjsmontMP} \cite{EjsmontTD}). It should be noted, that in \cite{BoBr2006} authors proved a free analog of the Lukacs theorem only in one direction, that is, if $\X$ and $\Y$ are free such that $\X+\Y$ and $(\X+\Y)^{-1/2}\X(\X+\Y)^{-1/2}$ are free, then $\X$ and $\Y$ have free Poisson distribution. It is still not known if free Poisson distributed $\X$ and $\Y$ have the required property.

In the previous paper \cite{SzpojanWesol} we investigated a free analogue of dual Lukacs regressions from \cite{BobWes2002Dual} in the case $(i,j)=(1,2)$. Our aim in this paper is to prove free analogues for the remaining cases. We will do this by developing technique from \cite{SzpojanWesol}.

The paper is organized as follows: Section 2 is devoted to give basics of free probability, in Section 3 we recall some facts from \cite{SzpojanWesol} and prove an auxiliary lemma, in Section 4 we give two characterizations of free Poisson and free binomial distributions, which are free analogues of the dual Lukacs regressions for $(i,j)\in \{(-1,1),\, (-2,-1)\}$.

\section{Preliminaries} We follow our previous paper \cite{SzpojanWesol}, to give basics of the free probability. More detailed introductions to the free probability can be found in \cite{NicaSpeicherLect} or \cite{VoiDykNica}. We will recall basic notions of non-commutative probability which are necessary for this paper.

A non-commutative probability space is a pair $(\mathcal{A},\varphi)$, where $\mathcal{A}$ is a unital algebra over $\C\,$ and $\varphi:\mathcal{A}\to\C\,$ is a linear functional, which is normalized, that is $\varphi(\I)=1$, where $\I$ is unit of algebra $\mathcal{A}$. Any element $\X$ of $\mathcal{A}$ is called a (non-commutative) random variable.

Let $H$ be a Hilbert space. By $\mathcal{B}(H)$ denote the space of bounded linear operators on $H$. For $\mathcal{A}\subset\mathcal{B}(H)$ and $\varphi$ a linear functional defined on $\mathcal{A}$ we say that $(\mathcal{A},\,\varphi)$ is a $W^*$-probability space when $\mathcal{A}$ is a von Neumann algebra and $\varphi$ is a normalized, faithfull and tracial state, that is $\varphi(\X^2)=0$ iff $\X=0$ and $\varphi(\X\,\Y)=\varphi(\Y\,\X)$ for any $\X,\Y\in\mathcal{A}$.

The $*$-distribution $\mu$ of a self-adjoint element $\X\in\mathcal{A}\subset\mathcal{B}(H)$ is a probabilistic measure on $\R$ such that $$\varphi(\X^r)=\int_{\R}\,t^r\,\mu(dt),\qquad \forall\,r=1,2,\ldots$$ In a setting of a general non-commutative probability space $(\mathcal{A},\varphi)$, we say that the distribution of the family $(\X_i)_{i=1,\ldots,q}$ is a linear functional $\mu_{\X_1,\ldots,\X_q}$ on the algebra $\C\,\langle x_1,\ldots,x_q\rangle$ of polynomials of non-commuting variables $x_1,\ldots,x_q$, defined by $$\mu_{\X_1,\ldots,\X_q}(P)=\varphi(P(\X_1,\ldots,\X_q))\qquad\forall\,P\in\C\,\langle x_1,\ldots,x_q\rangle.$$

Unital subalgebras $\mathcal{A}_i\subset \mathcal{A}$, $i=1,\ldots,n$, are said to be freely independent if $\varphi(\X_1,\ldots,\X_k)=0$ for $\X_j\in \mathcal{A}_{i(j)}$, where $i(j)\in\{1,\ldots,n\}$, such that $\varphi(\X_j)=0$, $j=1,\ldots,k$, if neighbouring elements are from different subalgebras, that is $i(1)\ne i(2)\ne \ldots \ne i(k-1)\ne i(k)$. Similarly, random variables $\X,\,\Y\in\mathcal{A}$ are free (freely independent) when subalgebras generated by $(\X,\,\I)$ and $(\Y,\,\I)$ are freely independent (here $\I$ denotes the identity operator).

For free random variables $\X$ and $\Y$ having distributions $\mu$ and $\nu$, respectively, the distribution of $\X+\Y$, denoted by $\mu\boxplus\nu$, is called free convolution of $\mu$ and $\nu$.

For self-adjoint and free $\X$, $\Y$ with distributions $\mu$ and $\nu$, respectively, and $\X$ positive, that is when the support of $\mu$ is a subset of $(0,\infty)$, free multiplicative convolution of $\mu$ and $\nu$ is defined as the distribution of $\sqrt{\X}\,\Y\sqrt{\X}$ and denoted by $\mu\boxtimes\nu$. Due to the tracial property of $\varphi$ the moments of $\Y\,\X$, $\X\,\Y$ and $\sqrt{\X}\,\Y\sqrt{\X}$ match.

Let $\chi=\{B_1,B_2,\ldots\}$ be a  partition of the set of numbers $\{1,\ldots,k\}$. A partition $\chi$ is a crossing partition if there exist distinct blocks $B_r,\,B_s\in\chi$ and numbers $i_1,i_2\in B_r$, $j_1,j_2\in B_s$ such that $i_1<j_1<i_2<j_2$. Otherwise $\chi$ is called a non-crossing partition. The set of all non-crossing partitions of $\{1,\ldots,k\}$ is denoted by $NC(k)$.

For any $k=1,2,\ldots$, cumulants of order $k$ are defined recursively as $k$-linear maps $\mathcal{R}_k:\mathcal{A}^k\to\C\,$ through equations
$$
\varphi(\Y_1\ldots\Y_m)=\sum_{\chi\in NC(m)}\,\prod_{B\in\chi}\,\mathcal{R}_{|B|}(\Y_i,\,i\in B)
$$
holding for any $\Y_i\in\mathcal{A}$, $i=1,\ldots,m$, and any $m=1,2,\ldots$,
with $|B|$ denoting the size of the block $B$.

Freeness can be characterized in terms of behaviour of cumulants in the following way: Consider unital subalgebras $(\mathcal{A}_i)_{i\in I}$ of an algebra $\mathcal{A}$ in a non-commutative probability space $(\mathcal{A},\,\varphi)$. Subalgebras $(\mathcal{A}_i)_{i\in I}$ are freely independent iff for any $n=2,3,\ldots$ and for any $\X_j\in\mathcal{A}_{i(j)}$ with $i(j)\in I$, $j=1,\ldots,n$ any $n$-cumulant
$$
\mathcal{R}_n(\X_1,\ldots,\X_n)=0
$$
if there exists a pair $k,l\in\{1,\ldots,n\}$ such that $i(k)\ne i(l)$.

In sequel we will use the following formula from \cite{BozLeinSpeich} which connects cumulants and moments for non-commutative random variables
\be\label{BLS}
\varphi(\X_1\ldots\X_n)=\sum_{k=1}^n\,\sum_{1<i_2<\ldots<i_k\le n}\,\mathcal{R}_k(\X_1,\X_{i_2},\ldots,\X_{i_k})\,\prod_{j=1}^k\,\varphi(\X_{i_j+1}\ldots\X_{i_{j+1}-1})
\ee
with $i_1=1$ and $i_{k+1}=n+1$ (empty products are equal 1).

The classical notion of conditional expectation has its non-commutative counterpart in the case when $(\mathcal{A},\varphi)$ is a  $W^*$-probability spaces, that is when $\mathcal{A}$ is necessarily a von Neumann algebra. Namely, if $\mathcal{B}\subset \mathcal{A}$ is a von Neumann subalgebra of the von Nuemann algebra $\mathcal{A}$, then there exists a faithful normal projection from $\mathcal{A}$ onto $\mathcal{B}$, denoted by $\varphi(\cdot|\mathcal{B})$, such that $\varphi(\varphi(\cdot|\mathcal{B}))=\varphi(\cdot)$. This projection $\varphi(\cdot|\mathcal{B})$ is a non-commutative conditional expectation given subalgebra $\mathcal{B}$. If $\X\in \mathcal{A}$ is self-adjoint then $\varphi(\X|\mathcal{B})$ defines a unique self-adjoint element in $\mathcal{B}$ satisfying the above equation. For $\X\in\mathcal{A}$ by $\varphi(\cdot|\X)$ we denote conditional expectation give then von Neumann subalgebra $\mathcal{B}$ generated by $\X$ and $\I$. Non-commutative conditional expectation has many properties analogous to those of classical conditional expectation. For more details one can consult e.g. \cite{Takesaki}. Here we state two of them we need in the sequel. The proofs can be found in \cite{BoBr2006}.
\begin{lemma}\label{conexp} Consider a $W^*$-probability space $(\mathcal{A},\varphi)$.
\begin{itemize}
\item If $\X\in\mathcal{A}$ and $\Y\in\mathcal{B}$, where $\mathcal{B}$ is a von Neumann subalgebra of $\mathcal{A}$, then
\be\label{ce1}
\varphi(\X\,\Y)=\varphi(\varphi(\X|\mathcal{B})\,\Y).
\ee
\item If $\X,\,\Z\in\mathcal{A}$ are freely independent then
\be\label{ce2}
\varphi(\X|\Z)=\varphi(\X)\,\mathbb{I}.
\ee
\end{itemize}
\end{lemma}

Now we introduce basic analytical tools used to deal with non-commutative random variables and their distributions.

For a non-commutative random variable $\X$ its $r$-transform is defined as
\be\label{rtr}
r_{\X}(z)=\sum_{n=0}^{\infty}\,\mathcal{R}_{n+1}(\X)\,z^n,
\ee
where $\mathcal{R}_{n}(\X)=\mathcal{R}_n(\X,\ldots,\X).$
In \cite{VoiculescuAdd} it is proved that $r$-transform of a random variable with compact support is analytic in a neighbourhood of zero.  From properties of cumulants it is immediate that for $\X$ and $\Y$ which are freely independent
\be\label{freeconv}
r_{\X+\Y}=r_{\X}+r_{\Y}.
\ee
This relation explicitly (in the sense of $r$-transform) defines free convolution of $\X$ and $\Y$.
If $\X$ has the distribution $\mu$, then often we will write $r_{\mu}$ instead $r_{\X}$.

Another analytical tool is an $S$-transform which works nicely with products of freely independent variables. For a noncommutative random variable $\X$ its $S$-transform, denoted by $S_{\X}$, is defined through the equation
\be\label{Str}
R_{\X}(zS_{\X}(z))=z,
\ee
where $R_{\X}(z)=zr_{\X}(z)$. For $\X$ and $\Y$ which are freely independent
\be\label{Scon}
S_{\X\,\Y}=S_{\X}\,S_{\Y}.
\ee

Cauchy transform of a probability measure $\nu$ is defined as
$$
G_{\nu}(z)=\int_{\R}\,\frac{\nu(dx)}{z-x},\qquad \Im(z)>0.
$$
Cauchy transforms and $r$-transforms are related by
\be\label{Crr}
G_{\nu}\left(r_{\nu}(z)+\frac{1}{z}\right)=z.
\ee

Finally we introduce moment generating function $M_{\X}$ of a random variable $\X$ by
\be\label{mgf}
M_{\X}(z)=\sum_{n=0}^{\infty}\,\varphi(\X^n)\,z^n.
\ee
Moment generating function and $S$-transform of $\X$ are related through
\be\label{MSr}
M_{\X}\left(\frac{z}{1+z}\,S_{\X}(z)\right)=z.
\ee
\section{Auxiliary results}
In this section we state some results, that we will use in the next section. 
First, let us recall definitions of distributions in which we will be interested further.

A non-commutative random variable $\X$ is said to be free-Poisson variable if it has Marchenko-Pastur (or free-Poisson) distribution $\nu=\nu(\lambda,\alpha)$ defined by the formula
\be\label{MPdist}
\nu=\max\{0,\,1-\lambda\}\,\delta_0+\lambda \tilde{\nu},
\ee
where $\lambda\ge 0$ and the measure $\tilde{\nu}$, supported on the interval $(\alpha(1-\sqrt{\lambda})^2,\,\alpha(1+\sqrt{\lambda})^2)$, $\alpha>0$ has the density (with respect to the Lebesgue measure)
$$
\tilde{\nu}(dx)=\frac{1}{2\pi\alpha x}\,\sqrt{4\lambda\alpha^2-(x-\alpha(1+\lambda))^2}\,dx.
$$
The parameters $\lambda$ and $\alpha$ are called the rate and the jump size, respectively.

Marchenko-Pastur distribution arises in a natural way  as an almost sure weak limit of empirical distributions of eigenvalues for random matrices of the form ${\bf X}\,{\bf X}^T$ where ${\bf X}$ is a matrix with zero mean iid entries with finite variance, in particular for Wishart matrices, (see \cite{MarchPastur}) and as a marginal distribution of a subclass of classical stochastic processes, called quadratic harnesses (see e.g. \cite{BrycWes2005}).

It is worth to note that a non-commutative variable with Marchenko-Pastur distribution arises also as a limit in law (in non-commutative sense) of variables with distributions $((1-\frac{\lambda}{N})\delta_0+\frac{\lambda}{N}\delta_{\alpha})^{\boxplus N}$ as $N\to\infty$, see \cite{NicaSpeicherLect}. Therefore, such variables are often called free-Poisson.

It is easy to see that if $\X$ is free-Poisson with distribution $\nu(\lambda,\alpha)$ then $\mathcal{R}_n(\X)=\alpha^n\lambda$, $n=1,2,\ldots$. Therefore its $r$-transform has the form
$$
r_{\nu(\lambda,\alpha)}(z)=\frac{\lambda\alpha}{1-\alpha z}.
$$

A non-commutative random variable $\Y$ is free-binomial if its distribution $\beta=\beta(\sigma,\theta)$ is defined by
\be\label{freebeta}
\beta=(1-\sigma)\mathbb{I}_{0<\sigma<1}\,\delta_0+\tilde{\beta}+(1-\theta)\mathbb{I}_{0<\theta<1}\delta_1,
\ee
where $\tilde{\beta}$ is supported on the interval $(x_-,\,x_+)$,
\begin{align}
\label{binom_supp}
x_{\pm}=\left(\sqrt{\frac{\sigma}{\sigma+\theta}\,\left(1-\frac{1}{\sigma+\theta}\right)}\,\pm\,\sqrt{\frac{1}{\sigma+\theta}\left(1-\frac{\sigma}{\sigma+\theta}\right)}\right)^2,
\end{align}
and has the density
$$
\tilde{\beta}(dx)=(\sigma+\theta)\,\frac{\sqrt{(x-x_-)\,(x_+-x)}}{2\pi x(1-x)}\,dx,
$$
where
$(\sigma,\theta)\in \left\{(\sigma,\theta):\,\frac{\sigma+\theta}{\sigma+\theta-1}>0,\,\frac{\sigma\theta}{\sigma+\theta-1}>0\right\}$.
The n-th free convolution power of distribution
$$
p\delta_0+(1-p)\delta_{1/n}
$$
is free-binomial distribution with parameters $\sigma=n(1-p)$ and $\theta=np$, which justifies the name of the distribution (see \cite{SaitohYosida}).

Its Cauchy transform is of the form (see e.g. the proof of Cor. 7.2 in \cite{CaCa})
\be\label{betaC}
G_{\sigma,\theta}(z)=\frac{(\sigma+\theta-2)z+1-\sigma-\sqrt{[(\sigma+\theta-2)z+1-\sigma]^2-4(1-\sigma-\theta)z(z-1)}}{2z(1-z)}.
\ee

Next we recall a result from \cite{SzpojanWesol}.

\begin{proposition}
\label{free_u_v}
Let $(\mathcal{A},\,\varphi)$ be a $W^*$-probability space. Let $\V$ and $\U$ in $\mathcal{A}$ be freely independent, such that $\V$ is free-Poisson with parameters $(\lambda,\,\alpha)$ and $\U$ is free-binomial with parameters $(\sigma,\theta)$, $\sigma+\theta=\lambda$. Define
\be\label{XY}
\X=\V^{\frac{1}{2}}\,\U\,\V^{\frac{1}{2}}\qquad\mbox{and}\qquad\Y=\V-\V^{\frac{1}{2}}\,\U\,\V^{\frac{1}{2}}.
\ee

Then $\X$ and $\Y$ are freely independent and their distributions are free-Poisson with parameters $(\theta,\alpha)$  and $(\sigma,\alpha)$, respectively.
\end{proposition}
If we additionally assume, that support of $\Y$ does not contain $0$, then von Neumann algebra generated by $\Y$ and $\I$ contains also $\Y^{-1}$ and $\Y^{-2}$, from which it follows, that if $\X$ and $\Y$ are free, so are $\X$ and $\Y^{-1}$ and also $\X$ and $\Y^{-2}$. \\
Now we can state the following observation.
\begin{proposition}
\label{coro}
Let $\U$ and $\V$ be freely independent random variables in a $W^*$-probability space. Assume that $\V$ is free-Poisson with parameters $\theta+\sigma>1$ and $\alpha$ and $\U$ is free-binomial with parameters $\sigma$ and $\theta>1$. Then
$$
\varphi\left(\left.\V-\V^{\frac{1}{2}}\,\U\,\V^{\frac{1}{2}}\right|\V^{\frac{1}{2}}\,\U\,\V^{\frac{1}{2}}\right)=\theta\alpha\,\I,
$$
$$
\varphi\left(\left.\left(\V-\V^{\frac{1}{2}}\,\U\,\V^{\frac{1}{2}}\right)^{-1}\right|\V^{\frac{1}{2}}\,\U\,\V^{\frac{1}{2}}\right)=\frac{1}{\alpha(\theta-1)}\,\I,
$$
$$
\varphi\left(\left.\left(\V-\V^{\frac{1}{2}}\,\U\,\V^{\frac{1}{2}}\right)^{-2}\right|\V^{\frac{1}{2}}\,\U\,\V^{\frac{1}{2}}\right)=\frac{\theta}{\alpha^2(\theta-1)^3}\,\I.
$$
\end{proposition}
\begin{proof}
First let us notice that the assumptions about parameters of distributions ensure that all above moments exist. It follows from the fact that these distributions have compact support, $0$ is not in the support of $\V$, $1$ is not in the support of $\U$.\\
From freeness noticed in Proposition \ref{free_u_v} and Lemma \ref{conexp} \eqref{ce2} it follows that conditional expectations are constant times identity, where constants are equal to suitable moments of free Poisson random variable.
\end{proof}
In the proof of the second theorem in the next section, we will need to compute free cumulants of the type $\mathcal{R}_n\left(\X^{-1},\X,\ldots,\X\right)$. By the moment-cumulant formula \eqref{BLS}, we see that such cumulants can be expressed in terms of free cumulants of the random variable $\X$ and the cumulant $\mathcal{R}_1\left(\X^{-1}\right)=\varphi\left(\X^{-1}\right)$. Next lemma gives recurrence relation between cumulants $\mathcal{R}_n\left(\X^{-1},\X,\ldots,\X\right)$, and cumulants of variable $\X$.
\begin{lemma}
\label{lem_cum}
Let $\V$ be compactly supported, invertible non-commutative random variable. Define $C_n=\mathcal{R}_{n}\left(\V^{-1},\V,\ldots,\V\right)$ and $C(z)=\sum_{i=1}^{\infty}C_iz^{i-1}$. Then for $z$ in a neighbourhood of $0$ we have
\be 
\label{lem_1}
C(z)=\frac{z+C_1}{1+zr(z)},
\ee
where $r$ is  an $r$-transform of $\V$. In particular,
\be 
\label{lem_2}
C_2=1-C_1\mathcal{R}_1(\V),\ \ C_n=-\sum_{i=1}^{n-1}C_i\mathcal{R}_{n-i}(\V),\, \quad n\geq 2.
\ee
\end{lemma}
\begin{proof}
It is evident that equations \eqref{lem_1} and \eqref{lem_2} are equivalent and we will proof only \eqref{lem_1}.
First, we observe that from the moment-cumulant formula \eqref{BLS} it follows that
\begin{align*}
\varphi\left(\V^n\right)=\varphi\left(\V^{-1}\V^{n+1}\right)=C_1\varphi\left(\V^{n+1}\right)+
\sum_{k=2}^{n+2}C_k\sum_{i_1+\ldots+i_k=n+2-k}\varphi\left(\V^{i_1}\right)\ldots\varphi\left(\V^{i_k}\right).
\end{align*}
Now we multiply left and right hand sides of the above equation by $z^n$ and sum over $n$ from $0$ to $\infty$, which results in
\small
\begin{align*}
M(z)&=\frac{C_1}{z}\left(M(z)-1\right)+\frac{1}{z}\sum_{n=0}^{\infty}
\sum_{k=2}^{n+2}C_kz^{k-1}\sum_{i_1+\ldots+i_k=n+2-k}\varphi\left(\V^{i_1}\right)z^{i_1}\ldots\varphi\left(\V^{i_k}\right)z^{i_k}\\
&=\frac{C_1}{z}\left(M(z)-1\right)+\frac{1}{z}\sum_{k=2}^{\infty}C_kz^{k-1}\sum_{n=k-2}^{\infty}\,\sum_{i_1+\ldots+i_k=n+2-k}\varphi\left(\V^{i_1}\right)z^{i_1}\ldots\varphi\left(\V^{i_k}\right)z^{i_k}\\
&=\frac{C_1}{z}\left(M(z)-1\right)+\frac{1}{z}\sum_{k=2}^{\infty}C_kz^{k-1}M^k(z)=\frac{C_1}{z}\left(M(z)-1\right)+\frac{1}{z}M(z)\left(C(zM(z))-C_1\right).
\end{align*}
\normalsize
We finally get
$$
zM(z)=M(z)C(zM(z))-C_1.
$$
Taking into account that $zM(z)=G\left(\frac{1}{z}\right)$ we obtain
\begin{align*}
G\left(\frac{1}{z}\right)&=\frac{1}{z}G\left(\frac{1}{z}\right)C\left(G\left(\frac{1}{z}\right)\right)-C_1,
\end{align*}
and thus
\begin{align*}
G(z)&=zG(z)C\left(G(z)\right)-C_1.
\end{align*}
From \eqref{Crr} we see that
\begin{align*}
z+C_1&=\left(r(z)+\frac{1}{z}\right)zC(z),
\end{align*}
and thus
\begin{align*}
C(z)&=\frac{z+C_1}{1+zr(z)}.
\end{align*}
\end{proof}
\section{Main results}
In this section we give two characterizations of free Poisson and free binomial distributions. The technique used in the proofs develops method from \cite{SzpojanWesol}. These two theorems below complete free probability analogues of classical results from \cite{BobWes2002Dual}.
\begin{theorem}
\label{th1}
Let $(\mathcal{A},\,\varphi)$ be a $W^*$-probability space and let $\U,\,\V$ be non-commutative variables in $(\mathcal{A},\,\varphi)$ which are freely independent, $\V$ has a distribution compactly supported in $(0,\infty)$ and  distribution of $\U$ is supported on $[0,1-\delta)$, for some $\delta\in(0,1)$. Assume that there exist real constants $c$ and $d$ such that
\be\label{reg1}
\varphi\left(\left.\V^{\frac{1}{2}}\left(\I-\U\right)\,\V^{\frac{1}{2}}\right|\V^{\frac{1}{2}}\,\U\,\V^{\frac{1}{2}}\right)=c\,\I
\ee
and
\be\label{reg2}
\varphi\left(\left.\V^{-\frac{1}{2}}\left(\I-\U\right)^{-1}\,\V^{-\frac{1}{2}}\right|\V^{\frac{1}{2}}\,\U\,\V^{\frac{1}{2}}\right)=d\,\I.
\ee
Then there exists $F>1$ such that $\V$ has free-Poisson distribution, $\nu(\lambda,\alpha)$ with $\lambda=\sigma+\theta$, $\left(\sigma=\frac{F-1}{cd-1},\theta=1+\frac{1}{c d-1}\right)$, $\alpha=\frac{c d-1}{d}$ and $\U$ has free-binomial distribution,  $\beta(\sigma,\theta)$.
\end{theorem}
\begin{proof}
By functional calculus in von Neumann algebras (see \cite{Takesaki}) $\V$ and $\I-\U$ are invertible since we assumed that the distributions of $\V$ has no atom at $0$ and the distribution of $\U$ is separated from $1$. Recall that in von Neumann algebras, spectral norm of normal element is equal to its norm. Since spectrum of $\U$ is by assumption contained in $[0,1-\delta)$, we have $\|\U\|<1$, and we can write $(\I-\U)^{-1}=\sum_{i=0}^{\infty}\U^i$. Since $\V$ is positive, so $\V^{-1}$ is also positive, and $\V^{-\frac{1}{2}}$ is well defined. The variables $\V^{-1}$ and $\V^{-\frac{1}{2}}$ belong to the von Neumann algebra generated by $\{\I,\V\}$, so are free with $\U$.
\\
We can multiply both sides of \eqref{reg1} and \eqref{reg2}, by $\left(\V^{\frac{1}{2}}\,\U\,\V^{\frac{1}{2}}\right)^n$ and by properties of conditional expectation in von Neumann algebras (Lemma \ref{conexp}) we get
\be\label{var1}
\varphi(\V(\V\U)^n)-\varphi((\V\U)^{n+1})=c\,\varphi((\V\U)^n)
\ee
and 
\be\label{var2}
\varphi\left(\left(\sum_{i=1}^\infty\U^i\right)\left(\V\U\right)^{n-1}\right)=d\,\varphi((\V\U)^n),
\ee
where \eqref{var1} is true for $n\geq 0$, and \eqref{var2} for $n\geq 1$.
Since the series $\sum_{i=1}^\infty\U^i$ converges absolutely, the second equation can be rewritten as 
\be
\sum_{i=1}^\infty\varphi\left(\U^i\left(\V\U\right)^{n-1}\right)=d\,\varphi((\V\U)^n).
\ee
Now we define sequences $(\alpha_n)_{n\ge 0}$, $(\beta_n)_{n\ge 0}$ and $(\gamma_{i,n})_{i,n\ge 0}$ as follows
\begin{align*}
\alpha_n=\varphi((\V\U)^n),\quad \beta_n=\varphi(\V(\V\U)^n),\quad\mbox{and}\quad \gamma_{i,n}=\varphi(\U^i(\V\U)^n),\quad i,n=0,1,\ldots.
\end{align*}
Then equations \eqref{var1} and \eqref{var2} have the form
\be\label{rew1}
\beta_n-\alpha_{n+1}=c\,\alpha_n
\ee
and
\be\label{rew2}
\sum_{i=1}^\infty\gamma_{i,n-1}=d\,\alpha_n.
\ee
Multiplying both sides of first equation by $z^n$ for $n\geq 0$ and taking sum over $n$ we obtain,
\begin{align}
\label{eq_first}
B(z)-\frac{1}{z}(A(z)-1)=c\,A(z),
\end{align}
where 
$$
M_{\U\V}(z)=A(z)=\sum_{n=0}^{\infty}\,\alpha_n\,z^n,\qquad B(z)=\sum_{n=0}^{\infty}\,\beta_n\,z^n.\qquad 
$$
From \cite{SzpojanWesol} (eq. 29 and 30) we know that we can express functions $A$ and $B$ as  
\be
\label{ab}
A(z)=1+zD(z)r(zD(z)), \ \ 
B(z)=zD(z)r^2(zD(z))+r(zD(z)),
\ee
where $r$ is r-transform of $\V$ and $D$ is generating function of sequence $\delta_n=\varphi\left(\U(\V\U)^n\right)$, $n\geq 0$.

Now we proceed to find similar form for equation \eqref{rew2}.
We multiply both sides of this equation by $z^n$ and sum over $n\geq 1$, which, by changing order of the summation results in
\be
\label{G1}
z\sum_{i=1}^\infty G_i(z)=d(A(z)-1),
\ee
where
$$
G_i(z)=\sum_{n=0}^\infty z^n\gamma_{i,n}.
$$
Note that by the Cauchy-Schwarz inequality and assumptions about support of $\U$ we have \begin{align*}|G_i(z)|&\leq \sum_{n=0}^\infty |z|^n|\varphi\left(\U^{2i}\right)|^{\frac{1}{2}}|\varphi\left((\U\V)^n(\V\U)^n\right)|^{\frac{1}{2}}\leq \\ &\leq 
\varphi(\U^{2i})^{1/2}\sum_{n=0}^\infty|z|^n|\varphi\left((\U\V)^n(\V\U)^n\right)|^{\frac{1}{2}}
\leq (1-\delta)^i\sum_{n=0}^\infty|z|^n|\varphi\left((\U\V)^n(\V\U)^n\right)|^{\frac{1}{2}}.
\end{align*} Note that for some $C>0$ we have $\|(\V\U)^n(\U\V)^n\|\leq\|\V\U\|^{2n}\leq C^n$. Taking into account that the support of a random variable is contained in its spectrum, we conclude that the function series $\sum_{n=0}^\infty|z|^n|\varphi\left((\U\V)^n(\V\U)^n\right)|^{\frac{1}{2}}$ converges in a neighbourhood of $0$. Since $|1-\delta|<1$ the left hand side of \eqref{G1} is also well defined in a neighbourhood of $0$.\\
By moment-cumulant formula \eqref{BLS} and freeness of $\U$ and $\V$ we get
\begin{align*}
\gamma_{i,n}&=\Rr_1\varphi(\U(\U\V)^{n-1}\U^{i})
\\
&+\Rr_2[\varphi(\U)\varphi(\U(\U\V)^{n-2}\U^{i})+
\varphi(\U\V\U)\varphi(\U(\V\U)^{n-3}\U^{i})+\ldots+ \varphi(\U(\V\U)^{n-2})\varphi(\U\U^{i})]\\
&+\ldots+\Rr_n\varphi^{n-1}(\U)\varphi(\U\U^i).
\end{align*}
Where $\Rr_n=\Rr_n(\V)$. In terms of $(\delta_n)$ and $(\gamma_{i+1,n})$ we can expand $\gamma_{i,n}$ as,
\begin{align*}
\gamma_{i,n}&=\Rr_1\gamma_{i+1,n-1}+\Rr_2\left(\delta_0\gamma_{i+1,n-2}+\delta_1\gamma_{i+1,n-3}+\ldots+\delta_{n-2}\gamma_{i+1,0}\right)\\
&+\ldots+\Rr_n\delta_0^{n-1}\gamma_{i+1,0}.
\end{align*}
Thus, for $n\geq 1$ and $i\geq 0$, we have, 
$$
\gamma_{i,n}=\sum_{k=1}^n\Rr_k\sum_{j_1+\ldots+j_k=n-k}\gamma_{i+1,j_1}\delta_{j_2}\ldots\delta_{j_k}.
$$
Hence for the function $G_{i}$ we obtain,
\begin{align*}
G_{i}(z)-\gamma_{i,0}&=\sum_{n=1}^{\infty}z^n\gamma_{i,n}
=\sum_{n=1}^{\infty}\sum_{k=1}^n\Rr_kz^k\sum_{j_1+\ldots+j_k=n-k}\gamma_{i+1,j_1}z^{j_1}\delta_{j_2}z^{j_2}\ldots\delta_{j_k}z^{j_k}\\
&=\sum_{k=1}^{\infty}\Rr_kz^k\sum_{n=k}^\infty\sum_{j_1+\ldots+j_k=n-k}\gamma_{i+1,j_1}z^{j_1}\delta_{j_2}z^{j_2}\ldots\delta_{j_k}z^{j_k}\\
&=\sum_{k=1}^{\infty}\Rr_kz^k\sum_{m=0}^\infty\sum_{j_1+\ldots+j_k=m}\gamma_{i+1,j_1}z^{j_1}\delta_{j_2}z^{j_2}\ldots\delta_{j_k}z^{j_k}\\
&=\sum_{k=1}^{\infty}\Rr_kz^kG_{i+1}(z)D^{k-1}(z)=zG_{i+1}(z)r(zD(z)).
\end{align*}
We sum the left and right hand sides of the above equation with respect to $i=0,1,\ldots$, which gives
\begin{align}
\label{gammaf}
\sum_{i=0}^\infty(G_i(z)-\gamma_{i,0})=zr(zD(z))\sum_{i=0}^\infty G_{i+1}(z).
\end{align}
Now we define a function $\Gamma(z)=\sum_{i=0}^\infty G_i(z)$, and constant $F=\sum_{i=0}^\infty\gamma_{i,0}$.\\
For the constant $F$ we have $$F=\sum_{i=0}^\infty\gamma_{i,0}=\sum_{i=0}^\infty\varphi\left(\U^i\right)=
\varphi\left((\I-\U)^{-1}\right)\in(1,\infty).$$ 
Equation \eqref{gammaf} can be rewritten as
$$
\Gamma(z)-F=zr(zD(z))\left(\Gamma(z)-G_{0}(z)\right).
$$
Taking into account that $G_{0}(z)=\sum_{n=0}^\infty z^n\varphi\left((VU)^n\right)=A(z)$, we can write
\be
\label{Gam}
\Gamma(z)=\frac{zr(zD(z))A(z)-F}{zr(zD(z))-1}.
\ee
We rewrite the left hand side of \eqref{G1} in terms of $\Gamma(z)$ and $A(z)$ obtaining
$$
z\left(\Gamma(z)-A(z)\right)=d(A(z)-1).
$$
Now we substitute $\Gamma(z)$  in the above equation by the right hand side of
\eqref{Gam} which gives
\be
\label{fin_form}
z\left(\frac{A(z)-F}{zr(zD(z))-1}\right)=d(A(z)-1).
\ee
After defining auxiliary function $h(z)=zD(z)r(zD(z))$ and using \eqref{ab}, we can rewrite the above equation and \eqref{eq_first} as a system
\begin{align}
\label{hDel2}
zD(z)\frac{h(z)+1-F}{h(z)-D(z)}&=dh(z),\\
\frac{h^2(z)+h(z)}{zD(z)}=\frac{h(z)}{z}&+c(h(z)+1).
\nonumber
\end{align}
We transform the above equations in order to have at the left hand sides only $h^2(z)$ and we compare right hand sides of these equations. After cancellations we arrive at 
\begin{align}
\label{hD2}
\frac{h(z)}{zD(z)}=\frac{\frac{F-1}{d}+c}{1-zD(z)(c-\frac{1}{d})}.
\end{align}
Recall that $h(z)=zD(z)r(zD(z))$, $\lim_{z\to 0}zD(z)=0$, and $r$ is analytic at $z=0$. Therefore we have
\begin{align}
\label{r2}
r(z)=\frac{\frac{F-1}{d}+c}{1-z(c-\frac{1}{d})}=\frac{\lambda\alpha}{1-z\alpha}.
\end{align}
It means that $\V$ has a free-Poisson distribution with parameters $\alpha=c-\frac{1}{d}$ and $\lambda=1+\frac{F}{cd-1}$. Note that by the Cauchy-Schwarz inequality $cd>1$ which implies $\alpha>0$ and $\lambda>1$; meaning that distribution of $\V$ does not have an atom at $0$.

It remains to prove that $\U$ has free binomial distribution. We will do this using $S$-transforms. Note that we can rewrite equation \eqref{hDel2} as
\begin{align*}
D(z)\left(z(h(z)+1-F)+\frac{F}{\alpha(\lambda-1)}h(z)\right)=h^2(z)\frac{F}{\alpha(\lambda-1)}.
\end{align*}
In the above equation we substitute $D(z)=\frac{h(z)}{\lambda\alpha z+\alpha z h(z)}$ which is equivalent to \eqref{hD2}. After canceling $h(z)$, which is allowed in a neighbourhood of $0$, we get
\begin{align*}
\frac{F}{\lambda-1}zh^2(z)+h(z)\left\{\left(\frac{\lambda F}{\lambda-1}-1\right)z-\frac{F}{\alpha(\lambda-1)}\right\}-z(1-F)=0.
\end{align*}
Recall that $h=M_{\U\V}$. We define $\Psi_{\U\V}=h^{-1}$ and we rewrite the above equation as
\begin{align*}
\frac{F}{\lambda-1}\Psi_{\U\V}(z)z^2+z\left\{\left(\frac{\lambda F}{\lambda-1}-1\right)\Psi_{\U\V}(z)-\frac{F}{\alpha(\lambda-1)}\right\}-\Psi_{\U\V}(z)(1-F)=0.
\end{align*}
Now we can use equation \eqref{MSr} to find $S$-transform of $\U\V$,
\begin{align*}
S_{\U\V}(z)=\frac{F}{\alpha\{(1-\lambda)(1-F)+zF\}}.
\end{align*} 
Since we know the distribution of $\V$ and $A=M_{\V\U}+1$, where $M_{\V\U}$ is moment generating function of $\V\U$, we can find the distribution of $\U$ by calculating S-transforms and using \eqref{Scon}.\\
Since the $S$-transform of $\V$ (being free-Poisson) is
\begin{align*}
S_\V(z)=\frac{1}{\alpha\lambda+\alpha z},
\end{align*}
by \eqref{Scon} we get
\begin{align*}
S_\U=1+\frac{F+\lambda-1}{(1-\lambda)(1-F)+Fz}.
\end{align*}
Now we use \eqref{Str} and \eqref{MSr} to find the corresponding Cauchy transform which appears to be
$$G_{\U}(z)=\frac{(\sigma+\theta-2)z+1-\sigma-\sqrt{[(\sigma+\theta-2)z+1-\sigma]^2-4(1-\sigma-\theta)z(z-1)}}{2z(1-z)},$$
where $\sigma =(\lambda-1)(1-\frac{1}{F})=\frac{F-1}{cd-1}, \ \ \theta =1+\frac{\lambda-1}{F}=\frac{cd}{cd-1}$. This is the Cauchy transform of the free binomial distribution with parameters $\sigma,\theta$; see \eqref{betaC}. By the Cauchy-Schwarz inequality $cd\geq 1$, so, $\theta\geq 1$, which means that the distribution of $\U$ does not have an atom at 1. It can be also checked by direct calculation that in this case $x_+$ from \eqref{binom_supp} is equal to $\frac{1}{(F+c d-1)^2}\left(\sqrt{cd(cd-1)}+\sqrt{F(F-1)}\right)^2$, which means that the continuous part of the support of distribution of $\U$ is separated from $1$.
\end{proof}

\begin{theorem}
Let $(\mathcal{A},\,\varphi)$ be a $W^*$-probability space and $\U,\,\V$ be non-commutative variables in $(\mathcal{A},\,\varphi)$ which are freely independent, $\V$ has a distribution compactly supported in $(0,\infty)$ and  distribution of $\U$ is supported on $[0,1-\delta)$ for some $\delta\in(0,1)$. Assume that there exist real constants $c$ and $d$ such that
\be\label{reg21}
\varphi\left(\left.\V^{-\frac{1}{2}}\left(\I-\U\right)^{-1}\,\V^{-\frac{1}{2}}\right|\V^{\frac{1}{2}}\,\U\,\V^{\frac{1}{2}}\right)=c\,\I
\ee
and
\be\label{reg22}
\varphi\left(\left.\V^{-\frac{1}{2}}\left(\I-\U\right)^{-1}\,\V^{-1}\left(\I-\U\right)^{-1}\,\V^{-\frac{1}{2}}\right|\V^{\frac{1}{2}}\,\U\,\V^{\frac{1}{2}}\right)=d\,\I.
\ee
Then there exists $F>1$, such that $\V$ has a free-Poisson distribution, $\nu(\lambda,\alpha)$ with $\lambda=\sigma+\theta$, $\left(\sigma=\frac{c^2 (F-1)}{d-c^2},\theta=\frac{d}{d-c^2}>1\right)$, $\alpha=\frac{d-c^2}{c^3}>0$ and $\U$ has a free-binomial distribution,  $\beta(\sigma,\theta)$.
\end{theorem}
\begin{proof}
We proceed with equation \eqref{reg21} similarly as in the proof of Theorem \ref{th1} and we obtain \eqref{fin_form} with $d$ replaced by $c$.
With equation \eqref{reg22} we proceed in a similar manner. First, we multiply both sides by $\left(\V^{\frac{1}{2}}\U\V^{\frac{1}{2}}\right)^n$ and by the tracial property  of $\varphi$ we obtain for $n\geq 1$
\begin{align*}
\varphi\left((\I-\U)^{-1}\V^{-1}(\I-\U)^{-1}\U(\V\U)^{n-1}\right)=d\varphi\left((\V\U)^n\right).
\end{align*} 
We can rewrite the above equation for $n\geq 1$ as
\begin{align}
\label{var21}
\sum_{i=0}^\infty\sum_{j=1}^\infty\varphi\left(\U^{i}\V^{-1}\U^{j}(\V\U)^{n-1}\right)=d\varphi\left((\V\U)^n\right).
\end{align} 
Now we define a triple sequence $\eta_{n,i,j}=\varphi\left((\V\U)^{n}\U^{i}\V^{-1}\U^{j}\right)$ for $n,i,j\geq 0.$ For every $i,j\geq 0$ let us define $N_{i,j}(z)=\sum_{n=0}^\infty \eta_{i,j,n}z^n$.\\
Multiplying \eqref{var21} by $z^n$ and summing up with respect to $n$ from $1$ to $\infty$ we get
\begin{align}
\label{fineq21}
z\sum_{i=0}^\infty\sum_{j=1}^\infty N_{i,j}(z)=d(A(z)-1).
\end{align}

We proceed with the moment-cumulant formula \eqref{BLS} in order to express the left hand side of \eqref{fineq21} in terms of functions $D$ and $h$, where $D,h$ are defined exactly as in the proof of Theorem \ref{th1}.\\
First, let us note that for $n\geq 1$ we get
\begin{align*}
\eta_{n,i,j}=&\mathcal{R}_1\varphi\left((\V\U)^{n-1}\U^{i}\V^{-1}\U^{j+1}\right)\\
+&\mathcal{R}_2\left\{\varphi\left(\U\right)\varphi\left((\V\U)^{n-2}\U^{i}\V^{-1}\U^{j+1}\right)+\ldots+
\varphi\left(\U(\V\U)^{n-2}\right)\varphi\left(\U^{i}\V^{-1}\U^{j+1}\right)\right\}\\
+&\mathcal{R}_3\left\{\varphi^2(\U)\varphi\left((\V\U)^{n-3}\U^i\V^{-1}\U^{j+1}\right)+\ldots
\varphi\left(\U(\V\U)^{n-3}\right)\varphi(\U)\varphi\left(\U^{i}\V^{-1}\U^{j+1}\right)\right\}\\
+&\ldots
+\mathcal{R}_n\varphi^{n-1}(\U)\varphi\left(\U^{i}\V^{-1}\U^{j+1}\right)\\
+&\varphi\left(\U^j\right)\left\{  \right.& \\
&C_2\varphi\left((\V\U)^{n-1}U^{i+1}\right)\\
+&C_3\left\{\varphi(\U)\varphi\left((\V\U)^{n-2}\U^{i+1}\right)+\ldots+
\varphi\left(\U(\V\U)^{n-2}\right)\varphi\left(\U^{i+1}\right)\right\}\\
+&\ldots
+C_{n+1}\varphi^{n-1}(\U)\varphi\left(\U^{i+1}\right)
\left.\right\},
\end{align*}
where $C_i=R_i(\V^{-1},\V,\ldots,\V)$ as in Lemma \ref{lem_cum}.
The above expression can be rewritten as 
\begin{align}
\label{etarek}
\eta_{n,i,j}=&\sum_{l=1}^{n}\mathcal{R}_i\sum_{k_1+\ldots+k_l=n-l}\eta_{k_1,i,j+1}
\delta_{k_2}\ldots\delta_{k_l}\\\nonumber
+&\varphi(\U^j)\sum_{l=2}^{n+1}C_l\sum_{k_1+\ldots+k_{l-1}=n-(l-1)}\gamma_{k_1,i+1}
\delta_{k_2}\ldots\delta_{k_{l-1}},
\end{align}
where  $\gamma_{i,n}=\varphi(\U^i(\V\U)^n),\quad i,n=0,1,\ldots$, are defined as in the proof of Theorem \ref{th1}.
For $i,j\geq 0$ we also have $\eta_{0,i,j}=\varphi\left(\U^i\V^{-1}\U^j\right)=C_1\varphi\left(\U^{i+j}\right)$.
Using equation \eqref{etarek} we are able to establish recurrence relation between functions $N_{i,j}$, $N_{i,j+1}$ and $G_{i+1}(z)$, where $G_{i}(z)=\sum_{n=0}^{\infty}\gamma_{i,n}z^n$, as follows
\begin{align*}
N_{i,j}(z)=&\sum_{n=0}^{\infty}z^n\eta_{n,i,j}\\=
&C_1\varphi\left(\U^{i+j}\right)+
\sum_{n=1}^{\infty}\sum_{l=1}^{n}z^n\mathcal{R}_l\sum_{k_1+\dots+k_l=n-l}\eta_{k_1,i,j+1}
\delta_{k_2}\dots\delta{k_l}\\
+&\varphi\left(\U^{j}\right)\sum_{n=1}^\infty\sum_{l=2}^{n+1}C_lz^n\sum_{k_1+\ldots+k_{l-1}=n-(l-1)}
\gamma_{k_1,i+1}\delta_{k_2}\ldots\delta_{k_{l-1}}\\
=&C_1\varphi\left(\U^{i+j}\right)+
\sum_{n=1}^{\infty}\sum_{l=1}^{n}\mathcal{R}_lz^{l}\sum_{k_1+\dots+k_l=n-l}\eta_{k_1,i,j+1}z^{k_1}
\delta_{k_2}z^{k_2}\dots\delta{k_l}z^{k_l}\\
+&\varphi\left(\U^{j}\right)\sum_{n=1}^\infty\sum_{l=2}^{n+1}C_lz^{l-1}\sum_{k_1+\ldots+k_{l-1}=n-(l-1)}
\gamma_{k_1,i+1}z^{k_1}\delta_{k_2}z^{k_2}\ldots\delta_{k_{l-1}}z^{k_{l-1}}\\
=&C_1\varphi(\U^{i+j})+zN_{i,j+1}(z)r(zD(z))+\varphi(\U^{j})\frac{G_{i+1}(z)}{D(z)}\left(C(zD(z))-C_1\right).
\end{align*}

We will compute the left hand side of \eqref{fineq21} using the above formula
\begin{align*}
\sum_{j=0}^\infty N_{i,j}(z)=&zr(zD(z))\left(\sum_{j=0}^\infty N_{i,j}(z)-N_{i,0}(z)\right)+\\+&\frac{G_{i+1}(z)}{D(z)}\left(C(zD(z))-C_1\right)
\sum_{j=0}^{\infty}\varphi\left(\U^j\right)+C_1\sum_{j=0}^\infty\varphi\left(\U^{i+j}\right).
\end{align*}
Taking sum over $i$ we obtain
\begin{align}
\label{sum_nij}
\left(1-zr(zD(z))\right)\sum_{i=0}^{\infty}\sum_{j=0}^\infty N_{i,j}(z)=&-z\mathcal{R}(zD(z))\sum_{i=0}^\infty N_{i,0}+\\+&\frac{\sum_{i=0}^\infty G_{i+1}(z)}{D(z)}\left(C(zD(z))-C_1\right)
\sum_{j=0}^{\infty}\varphi\left(\U^j\right)\nonumber+\\+&C_1\sum_{i=0}^\infty
\sum_{j=0}^\infty\varphi\left(\U^{i+j}\right).\nonumber
\end{align}
Since we have \begin{align*}|\eta_{n,i,j}|=&|\varphi\left((\V\U)^n\U^i\V^{-1}\U^j\right)|\leq\varphi\left(\U^{2i}V^{-1}\U^{2j}\V^{-1}\right)^{1/2}\varphi\left((\V\U)^n(\U\V)^n\right)^{1/2}
\\ \leq&
(1-\delta)^{i+j} C,
\end{align*} for some $C>0$, convergence of the left hand side of \eqref{sum_nij} for $z$ in neighbourhood of $0$ can be proved similarly as it was done for the function $\Gamma$ in the proof of Theorem \ref{th1}.

Note that for $n>0$ we have
\begin{align*}
\eta_{n,i,0}&=\varphi\left((\V\U)^n\U^i\V^{-1}\right)=
\varphi\left(\U(\V\U)^{n-1}\U^{i+1}\right)=\gamma_{n-1,i+1},
\end{align*}
moreover $\eta_{0,i,0}=\varphi\left(\U^i\V^{-1}\right)=C_1\varphi\left(\U^i\right)$.
\\
Using these relations we can express $N_{i,0}$ in terms of $G_{i+1}$ as
\begin{align*}
N_{i,0}(z)=\sum_{n=0}^\infty z^n\eta_{n,i,0}=\eta_{0,i,0}+\sum_{n=1}^\infty z^n\gamma_{n-1,i+1}=C_1\varphi\left(\U^i\right)+zG_{i+1}(z).
\end{align*}
Define now
\begin{align*}\Gamma(z)=\sum_{i=0}^\infty G_i(z), &\hspace{0.5cm} F=\sum_{i=0}^\infty\varphi\left(\U^i\right)=\varphi\left((\I-\U)^{-1}\right)<\infty,
\\ H=&\sum_{i=0}^\infty
\sum_{j=0}^\infty\varphi\left(\U^{i+j}\right)=\varphi\left((\I-\U)^{-2}\right)<\infty.
\end{align*}
Taking into account that $G_0(z)=A(z)$ we can rewrite 
\eqref{sum_nij} as 
\begin{align*}
\sum_{i=0}^{\infty}\sum_{j=0}^\infty N_{i,j}(z)=&\frac{1}{1-zr(zD(z))}\Big\{-zr(zD(z))\left(C_1F+z(\Gamma(z)-A(z))\right)+\\+&\frac{\Gamma(z)-A(z)}{D(z)}\left(C(zD(z))-C_1\right)F
+C_1H\Big\}.
\end{align*}
Note that the sum over $j$ at the left hand side of \eqref{fineq21} begins with $j=1$. By subtracting from both sides of above equation $\sum_{i=0}^\infty N_{i,0}(z)=C_1F+z(\Gamma(z)-A(z))$, we obtain
\begin{align*}
\sum_{i=0}^{\infty}\sum_{j=1}^\infty N_{i,j}(z)=&\frac{1}{1-zr(zD(z))}\Big\{-z(\Gamma(z)-A(z))+\\+&\frac{\Gamma(z)-A(z)}{D(z)}\left(C(zD(z))-C_1\right)F
+C_1(H-F)\Big\}.
\end{align*}
Finally, we conclude that initial equations (\eqref{reg21} and \eqref{reg22})
transforms into
\begin{align*}
z\left(\Gamma(z)-A(z)\right)&=c(A(z)-1)\\
\frac{z}{1-zr(zD(z))}\Big\{-z(\Gamma(z)-A(z))+\\+\frac{\Gamma(z)-A(z)}{D(z)}\left(C(zD(z))-C_1\right)F
+C_1(H-F)\Big\}&=d(A(z)-1).
\end{align*}
Multiplying both sides of equations \eqref{reg21} and \eqref{reg22} just by $\V^{\frac{1}{2}}\U\V^{\frac{1}{2}}$ and applying $\varphi$, it is easy to see that $C_1(H-F)=\frac{d(F-1)}{c}$. Putting this into the above system and simplifying the second equation using the first, we obtain
\begin{align*}
z\left(\Gamma(z)-A(z)\right)=c(A(z)-1)&\\
\frac{z}{1-zr(zD(z))}\Big\{-c(A(z)-1)+\frac{c(A(z)-1)}{zD(z)}\left(C(zD(z))-C_1\right)F
+\frac{d(F-1)}{c}\Big\}=&\\=d(A(z)-1)&.
\end{align*}
If we define a function $h$ by $h(z)=zD(z)r(zD(z))$ as in the previous poof, after simple transformations, using \eqref{ab} we can rewrite the above system of equations as
\begin{align}
\label{r31}
zD(z)( h(z)+1-F)&=ch(z)(h(z)-D(z)),\\
\label{r32}
zD(z)\left\{-ch(z)+\frac{cFh(z)}{zD(z)}
\frac{zD(z)-C_1h(z)}{1+h(z)}+\frac{d(F-1)}{c}\right\}&=dh(z)(D(z)-h(z)).
\end{align}
Now we compare the left hand sides of these equations. After canceling $zD(z)$ which is allowed in some neghbourhood of $0$, we obtain
\begin{align*}
-ch(z)+\frac{cFh(z)}{zD(z)}
\frac{zD(z)-C_1h(z)}{1+h(z)}+\frac{d(F-1)}{c}=-\frac{d}{c}( h(z)+1-F).
\end{align*}
In \eqref{r32} we can first reduce the term without $h(z)$ then we can cancel one $h(z)$, which is allowed in neighbourhood of $0$, and thus we conclude that
\begin{align}
\frac{h(z)}{zD(z)}=\frac{\frac{1}{cC_1F}(cF+\frac{d}{c}-c)}{1-(\frac{d}{c}-c)\frac{1}{cC_1F}zD(z)}.
\end{align}
Note that by freeness $c=\varphi\left(\V^{-1}(\I-\U)^{-1}\right)=\varphi\left(\V^{-1}\right)
\varphi\left((\I-\U)^{-1}\right)=C_1F$,
if we define $\alpha=\frac{d-c^2}{c^2C_1F}=\frac{d-c^2}{c^3}$ and $\lambda=1+\frac{c^2F}{d-c^2}$, we can rewrite the last equation as
\begin{align}
\label{hD3}
\frac{h(z)}{zD(z)}=\frac{\lambda\alpha}{1-\alpha zD(z)}.
\end{align} 
Taking into account that $\lim_{z\to 0}zD(z)=0$, $h(z)=zD(z)r(zD(z))$ and that $r$ is analytic at $0$ we obtain
\begin{align*}
r(z)=\frac{\lambda\alpha}{1-\alpha z},
\end{align*}
which means that $\V$ has the free Poisson distribution with parameters $\lambda,\alpha$. Note that by the Cauchy-Schwarz inequality $d> c^2$, so we have $\alpha> 0$ and $\lambda> 1$ which means that $\V$ does not have an atom at $0$.

The distribution of $\U$ can be determined exactly in the same way as in the proof of Theorem \ref{th1}. One can see that equations \eqref{hD3} and \eqref{r31} are equivalent to equations \eqref{hDel2} and \eqref{r2} it allows to find $S$-transform of $\U$. So we conclude that $\U$ has a free binomial distribution with parameters $\sigma=\frac{(F-1)(\lambda-1)}{F}=\frac{(F-1)c^2}{d-c^2}$ and $\theta=\frac{F+\lambda-1}{F}=\frac{d}{d-c^2}$. Note that $\theta>1$ which means that distribution of $\U$ does not have an atom at $1$. Since $F>1$ we also have $\sigma>0$ and by direct calculation one can see that the continous part of the distribution of $\U$ is separated from $1$, in particular, $x_+=\frac{1}{(d+c^2(F-1))^2}\left(\sqrt{d(d-c^2)}+\sqrt{c^4F(F-1)}\right)^2$.
\end{proof}
\subsection*{Acknowledgement} The author thanks J. Weso\l{}owski for many helpful comments and discussions. This research was partially supported by NCN
grant 2012/05/B/ST1/00554.
\bibliographystyle{plain}
\bibliography{Bibl}
\end{document}